\documentclass[11pt,a4paper]{article}
\usepackage{amsmath,bm,graphicx,booktabs}
\usepackage{geometry,indentfirst,float,appendix,lmodern}
\usepackage{multirow}
\usepackage[authoryear]{natbib}
\usepackage{algorithm,algorithmic}
\usepackage[colorlinks,citecolor=blue,urlcolor=blue]{hyperref}
\usepackage{inputenc}
\usepackage[english]{babel}
\usepackage{amsthm}
\usepackage{color}
\usepackage{courier}
\usepackage{amssymb}
\usepackage{mathrsfs}
\usepackage{enumitem}
\usepackage{caption}

\providecommand{\keywords}[1]{\textbf{\textit{key words: }} #1}
\newcommand{\pcite}[1]{\citeauthor{#1}'s \citeyearpar{#1}}
\newtheorem{theorem}{Theorem}

\newtheorem{remark}{Remark}

\title{Analysis of the P\'olya-Gamma block Gibbs sampler for Bayesian logistic linear mixed models}

\date{November, 2017}

\author{ Xin Wang\thanks{Email: xinwang@iasate.edu}  \, and Vivekananda Roy\thanks{Email: vroy@iastate.edu} \\
	Department of Statistics, Iowa State University, Ames, IA}

\begin{document}
\maketitle

\begin{abstract}
In this article,  we construct a two-block Gibbs sampler using \pcite{polson2013bayesian} data augmentation technique with P\'olya-Gamma latent variables for Bayesian logistic linear mixed models under proper priors. Furthermore, we prove the uniform ergodicity of this Gibbs sampler, which guarantees the existence of the central limit theorems for MCMC based estimators. 
\end{abstract}

\section{Introduction}
\label{sec_intro}
Consider the logistic linear mixed model set-up \citep{McCullochCharlesE2008Glam, mcculloch2003generalized}.  Let  $(Y_1, Y_2, \dots, Y_N)$ denote the vector of $\text{Binomial}(n_i,p_i)$ random variables,  $\bm{x}_i$ and $\bm{z}_i$ be the $p \times 1$ and $q \times 1$ known covariates and random effect design vectors respectively associated with the $i$th observation for $i=1,\dots,N$. Let $\bm{\beta} \in \mathbb{R}^p$ be the unknown vector of regression coefficients and $\bm{u} \in \mathbb{R}^q$ be the random effects vector. Assume that $p_i = F(\bm{x}_i^T\bm{\beta} + \bm{z}_i^T\bm{u})$, where $F$ is the standard logistic distribution function, that is $F(t) \equiv e^t/(1+e^t)$ for $t\in \mathbb{R}$. Suppose we have $r$ random effects with $\bm{u} = (\bm{u}_1^T,\dots, \bm{u}_r^T)^T$, where $\bm{u}_j$ is a $q_j\times1$ vector with $q_j >0$, $q_1+\cdots + q_r = q$, and $\bm{u}_j \stackrel{\text{ind}}{\sim} N(0, \bm{I}_{q_j}1/\tau_j)$, where $\tau_j\in \mathbb{R}_+ = (0,\infty)$ is the precision parameter associated with $\bm{u}_j$ for $j=1,\dots,r$. The joint distribution of $\bm{u}$ is $N\left(0,\bm{D}(\bm{\tau})^{-1}\right)$, where $\bm{D}(\bm{\tau})=\oplus_{j=1}^{r}\tau_{j}\bm{I}_{q_{j}}$, $\bm{\tau}= (\tau_1, \dots, \tau_r)$ and $\oplus$ denotes the direct sum. The data model for the logistic linear mixed model is
\begin{eqnarray}
\label{eq:data}
Y_{i} | \bm{\beta}, \bm{u} & \overset{ind}\sim & \text{Binomial}(n_i, p_{i}) \text{ for } i=1,\dots,N \;\text{ with} \nonumber\\
p_{i} & = & F(\bm{x}_{i}^{T}\bm{\beta}+\bm{z}_{i}^{T}\bm{u}) \text{ for }i=1,\dots,N, \\
\bm{u}_{j}| \tau_j & \overset{\text{ind}}\sim & N\Big(0,\frac{1}{\tau_{j}}\bm{I}_{q_{j}}\Big) , \, j=1,\dots,r .\nonumber
\end{eqnarray}
Let $\bm{y} = (y_1, y_2, \dots, y_N)^T$ be the observed Binomial response variables. The likelihood function for $(\bm{\beta}, \bm{\tau})$ is
\begin{equation}
\label{eq:likelihood}
L\left(\bm{\beta},\bm{\tau}|\bm{y}\right)=\int_{\mathbb{R}^{q}}\prod_{i=1}^{N} {n_i \choose y_i} \frac{\left[\exp\left(\bm{x}_{i}^{T}\bm{\beta}+\bm{z}_{i}^{T}\bm{u}\right)\right]^{y_{i}}}{\left[1+\exp\left(\bm{x}_{i}^{T}\bm{\beta}+\bm{z}_{i}^{T}\bm{u}\right)\right]^{n_{i}}}\phi_{q}\left(\bm{u};\bm{0},\bm{D}(\bm{\tau})^{-1}\right)d\bm{u},
\end{equation}
where $\phi_{q}(s;a,B)$ is the probability density function of the $q$-dimensional normal distribution with mean vector $a$ and covariance matrix $B$ evaluated at $s$. In Bayesian framework, let $\pi\left(\bm{\beta}\right)$ and $\pi\left(\bm{\tau}\right)$ be the prior densities for $\bm{\beta}$ and $\bm{\tau}$ respectively. Assume that $\bm{\beta}$ and $\bm{\tau}$ are apriori independent. The joint posterior density of $\left(\bm{\beta},\bm{\tau}\right)$ is
\begin{equation}
\label{eq:marginal}
\pi\left(\bm{\beta},\bm{\tau}|\bm{y}\right)=\frac{1}{c_0\left(\bm{y}\right)}L\left(\bm{\beta},\bm{\tau}|\bm{y}\right)\pi\left(\bm{\beta}\right)\pi\left(\bm{\tau}\right),
\end{equation}
where  $c_0(\bm{y})$ is the marginal density of $\bm{y}$ with $
c_0\left(\bm{y}\right)=\int_{\mathbb{R}_{+}^{r}}\int_{\mathbb{R}^{p}}L\left(\bm{\beta},\bm{\tau}|\bm{y}\right)\pi\left(\bm{\beta}\right)\pi\left(\bm{\tau}\right)d\bm{\beta}d\bm{\tau}$.

The posterior density \eqref{eq:marginal} is intractable for any choice of the prior distributions of $\bm{\beta}$ and $\bm{\tau}$.  Generally, Markov chain Monte Carlo (MCMC) algorithms are used for exploring these posterior densities. Even in the absence of random effects, for generalized linear models, MCMC algorithms are needed to summarize the associated posterior densities. For probit regression models, \cite{albert1993bayesian} proposed a widely used data augmentation (DA) algorithm to sample from the corresponding posterior distributions. \cite{roy2007convergence} and \cite{chakraborty2017convergence} proved the geometric ergodicity of this DA algorithm for Bayesian probit regression model under improper and proper priors respectively. \cite{wang2017convergence} recently extended the convergence rate analysis of the block Gibbs samplers based on this DA technique for Bayesian probit linear mixed models under both proper and improper priors.

For logistic regression models, there have been several attempts for producing a DA algorithm similar to \pcite{albert1993bayesian} algorithm for the probit regression model (see e.g. \cite{holmes2006bayesian} and \cite{fruhwirth2010data}). Unfortunately, these algorithms are far more complex than \pcite{albert1993bayesian} algorithm. Only recently, \cite{polson2013bayesian} produced such a DA algorithm for logistic regression models using P\'olya-Gamma latent variables. \cite{choi2013polya} proved uniform ergodicity of the P\'olya-Gamma  DA Markov chain under normal priors on the regression parameters. \cite{choi2017analysis} showed that the Markov operator based on \pcite{polson2013bayesian} DA algorithm for one-way logistic ANOVA model is trace-class, which implies that the associated Markov Chain is geometrically ergodic. Both \cite{choi2013polya} and \cite{choi2017analysis} considered the special case when the data are binary, that is $n_i = 1$ for all $i$. However, there is no result in the literature about convergence analysis of any Gibbs samplers for Bayesian logistic linear mixed models. In this article, we construct a two-block Gibbs sampler for Bayesian logistic linear mixed models with normal priors on regression parameters and truncated Gamma priors on precision parameters.  We further establish uniform ergodicity of this Gibbs sampler. 

The article is organized as follows. In section \ref{sec_gibbs}, we
construct the two-block Gibbs sampler for the Bayesian logistic linear
mixed model under proper priors. In section \ref{sec_proof}, we prove the uniform ergodicity of the underlying Markov chain. Finally, we have some discussions in section \ref{sec_dis}.

\section{Two-block Gibbs sampler}
\label{sec_gibbs}


In \cite{polson2013bayesian}, a logistic linear mixed model example is introduced. In their example, normal distribution is used as the prior for regression coefficients and Gamma distribution is used as the prior for precision parameters. We assume the following priors: $\bm{\beta} \sim  N_p(\bm{Q}^{-1}\bm{\mu}_0, \bm{Q}^{-1})$ for some $p\times p$ positive definite matrix $\bm{Q}$ and $\bm{\mu}_0 \in \mathbb{R}^p$,
$\tau_j \overset{ind} \sim \text{truncated Gamma}(a_j,b_j,\tau_0), \,j=1,\dots,r$, where $b_j >0$ for $j=1,\dots, r$.  The density function of $\text{truncated Gamma}(a_j, b_j, \tau_0)$ is
\begin{equation}
\label{eq_tngamma0}
f\left(\tau_{j}|a_j,b_j,\tau_0\right)=[c\left(\tau_{0},a_{j},b_{j}\right)]^{-1}\tau_{j}^{a_{j}-1}\exp\left( - b_j \tau_j\right) I(\tau_j\geq\tau_0),
\end{equation}
where $c\left(\tau_{0},a_{j},b_{j}\right) = \int_{\tau_0}^{\infty}\tau^{a_{j}-1}\exp\left( - b_j \tau\right)  d\tau$ and $\tau_0 >0$ is a known constant.

 By Theorem 1 in \cite{polson2013bayesian}, 
\begin{align}
\label{eq:poly}
\frac{\left[\exp\left(\bm{x}_{i}^{T}\bm{\beta}+\bm{z}_{i}^{T}\bm{u}\right)\right]^{y_{i}}}{\left[1+\exp\left(\bm{x}_{i}^{T}\bm{\beta}+\bm{z}_{i}^{T}\bm{u}\right)\right]^{n_{i}}} & =2^{-n_{i}}\exp\left[\kappa_{i}\left(\bm{x}_{i}^{T}\bm{\beta}+\bm{z}_{i}^{T}\bm{u}\right)\right] \nonumber\\
 & \times \int_{0}^{\infty}\exp\left[-\omega_{i}\left(\bm{x}_{i}^{T}\bm{\beta}+\bm{z}_{i}^{T}\bm{u}\right)^{2}/2\right]p\left(\omega_{i}\right)d\omega_{i},
\end{align}
where $p\left(\omega_{i}\right)$ is the probability density function of the random variable $\omega_{i}\sim\text{PG}\left(n_i,0\right)$ and  $\kappa_{i}=y_{i}-n_i/2$ for $i=1,\dots,N$. Here, $\text{PG}(n_i,0)$ denotes the P\'olya-Gamma distribution with parameters $n_i$ and 0 with density \[
f(x\vert n_i,0) =\frac{2^{n_i-1}}{\Gamma(n_i)} \sum_{n=0}^{\infty} (-1)^n \frac{\Gamma(n+n_i)}{\Gamma(n+1)} \frac{(2n+n_i)}{\sqrt{2\pi x^3}} e^{-\frac{(2n+n_i)^2}{8x}}, \, x>0. \]
Let $\bm{\omega} = (\omega_1,\dots, \omega_{N})$ and the joint (posterior) density of $\bm{\beta},\bm{u},\bm{\omega}$ and $\bm{\tau}$ be 
\begin{align}
\label{eq:joint}
&\pi\left(\bm{\beta},\bm{u},\bm{\omega},\bm{\tau}|\bm{y}\right)  \propto \prod_{i=1}^{N}\exp\left[\kappa_{i}\left(\bm{x}_{i}^{T}\bm{\beta}+\bm{z}_{i}^{T}\bm{u}\right)-\omega_{i}\left(\bm{x}_{i}^{T}\bm{\beta}+\bm{z}_{i}^{T}\bm{u}\right)^{2}/2\right]p\left(\omega_{i}\right) \nonumber\\
& \times\phi_{q}\left(\bm{u};\bm{0},\bm{D(\bm{\tau})}^{-1}\right)\phi_{p}\left(\bm{\beta};\bm{Q}^{-1}\bm{\mu}_{0},\bm{Q}^{-1}\right)\prod_{j=1}^{r}  \tau_j^{a_j-1}e^{-b_j\tau_j}I(\tau_j \geq \tau_0).
\end{align}
From \eqref{eq:likelihood}, \eqref{eq:marginal} and \eqref{eq:poly}, it follows that $
\int_{\mathbb{R}^q}\int_{\mathbb{R}_{+}^N} \pi(\bm{\beta},\bm{u}, \bm{\omega},\bm{\tau}|\bm{y})d\bm{\omega} d\bm{u} = \pi(\bm{\beta},\bm{\tau}|\bm{y})$, 
which is our target posterior density. Using draws from all full conditional distribution distributions of \eqref{eq:joint}, we can run a Gibbs sampler with stationary density \eqref{eq:joint}. It is known that by combining and simultaneously drawing multiple parameters, the convergence of the Gibbs sampler can be improved \citep{liu1994covariance}, although the ``blocking" to be computationally efficient, the corresponding joint conditional distributions need to be tractable. Here we construct a two-block Gibbs sample for \eqref{eq:joint}. Let $\bm{\eta}=\left(\bm{\beta}^{T},\bm{u}^{T}\right)^{T}$, $\bm{\kappa} = (\kappa_{1},\dots,\kappa_{n})$ and $\bm{M}=\left(\bm{X},\bm{Z}\right)$
with $i$th row $\bm{m}_{i}^T$ and $\bm{\Omega}$ be the $n\times n$ diagonal matrix
with $i$th diagonal element $\omega_{i}$. Standard calculations show
that the conditional density of $\bm{\eta}$ is
\[
\pi\left(\bm{\eta}|\bm{\omega},\bm{\tau},\bm{y}\right)\propto\exp\left[-\frac{1}{2}\bm{\eta}^{T}\bm{M}^{T}\bm{\Omega}\bm{M}\bm{\eta}+\bm{\eta}^{T}\bm{M^T\kappa}\right]\exp\left[-\frac{1}{2}\bm{\eta}^{T}\bm{A}\left(\bm{\tau}\right)\bm{\eta}+\bm{\eta}^{T}\bm{l}\right],
\]
where $\bm{l}=\left(\bm{\mu}_{0}^{T},\bm{0}_{1\times q}\right)^{T}$, and $\bm{A}\left(\bm{\tau}\right) = \bm{Q} \oplus \bm{D}(\bm{\tau})$. That is,  \begin{equation}
\label{eq_eta}
\bm{\eta}|\bm{\omega},\bm{\tau},\bm{y}\sim N\left(\bm{\Sigma}^{-1}\bm{\mu},\bm{\Sigma}^{-1}\right)
\end{equation}
\[
\text{where }
\bm{\Sigma}=\left(\begin{array}{cc}
\bm{X}^{T}\bm{\Omega}\bm{X}+\bm{Q} & \bm{X}^{T}\bm{\Omega}\bm{Z}\\
\bm{Z}^{T}\bm{\Omega}\bm{X} & \bm{Z}^{T}\bm{\Omega}\bm{Z}+\bm{D}(\bm{\tau})
\end{array}\right)=\bm{M}^{T}\bm{\Omega}\bm{M}+\bm{A}\left(\bm{\tau}\right),\quad\bm{\mu}=\bm{M}^{T}\bm{\kappa}+\bm{l}.
\]
Similarly, the conditional density of $\left(\bm{\omega},\bm{\tau}\right)$ is
\begin{align*}
\pi\left(\bm{\omega},\bm{\tau}|\bm{\beta},\bm{u},\bm{y}\right) & \propto\prod_{i=1}^{N}\exp\left[-\omega_{i}\left(\bm{x}_{i}^{T}\bm{\beta}+\bm{z}_{i}^{T}\bm{u}\right)^{2}/2\right]p\left(\omega_{i}\right)\\
& \times\prod_{j=1}^{r}\tau_{j}^{a_{j}+q_{j}/2-1}\exp\left[-\left(\bm{u}_{j}^{T}\bm{u}_{j}/2+b_{j}\right)\tau_{j} \right] I(\tau_j \geq \tau_0).
\end{align*}
So, given $\bm{\beta},\bm{u}, \bm{y}$, we have that $\bm{\omega}$ and $\bm{\tau}$ are conditionally independent with 
\begin{align*}
\omega_{i}|\bm{\eta},\bm{y} & \overset{ind} \sim PG\left(n_i,|\bm{m}_{i}^{T}\bm{\eta}|\right),\, i=1,\dots,N,\\
\tau_{j}|\bm{\eta},\bm{y} & \overset{ind} \sim\text{truncated Gamma}\left(a_{j}+\frac{q_{j}}{2},b_{j}+\frac{\bm{u}_{j}^{T}\bm{u}_{j}}{2}, \tau_0\right), \, j=1,\dots,r.
\end{align*}

\begin{remark}
	\label{rem:tau0}
As in \cite{wang2017convergence}, we assume that the prior distribution for $\tau_j$ is a truncated Gamma distribution. However, while implementing the block Gibbs sampler in practice, a number slightly larger than the machine precision zero can be treated as $\tau_0$, practically avoiding the need to use any rejection sampling algorithms to draw from the truncated
	conditional distribution of $\bm{\tau}$. 
\end{remark}

Thus, one single iteration of the block Gibbs sampler $\{\bm{\eta}^{(m)}, \bm{\omega}^{(m)}, \bm{\tau}^{(m)}\}_{m=0}^{\infty}$ has the following two steps:

\begin{algorithm}[H]
	\caption*{{\bf{Algorithm}:} The $(m+1)$st iteration for the two-block Gibbs sampler}
	\begin{algorithmic}[1]
		\STATE Draw $\bm{\tau}_j^{(m+1)}$ from $\text{truncated Gamma}\left(a_{j}+q_{j}/2,b_{j}+\bm{u}_{j}^{T}\bm{u}_{j}/2,\tau_0\right)$ with $\bm{u} =\bm{u}^{(m)}$ for $j=1,\dots, r$, and independently draw $\omega_{i}^{(m+1)} \overset{\text{ind}}\sim  PG\left(n_i,|\bm{m}_{i}^{T}\bm{\eta}^{(m)}|\right)$ for  $i=1,\dots, N$.
		\STATE Draw $\bm{\eta}^{(m+1)}$ from (\ref{eq_eta}),  $\bm{\eta}^{(m+1)} \sim N_{p+q}\left( \bm{\Sigma}^{(m)-1}\left(\bm{M}^{T}\bm{\kappa}+\bm{l}\right),\bm{\Sigma}^{(m)-1}\right),$
where $\bm{\Sigma}^{(m)} = \bm{M}^{T}\bm{\Omega}^{(m+1)}\bm{M}+\bm{A}(\bm{\tau}^{(m+1)})$  and the diagonal elements of $\bm{\Omega}^{(m+1)}$ are $\omega_{i}^{(m+1)}$, $i=1,\dots,N$.
	\end{algorithmic}
\end{algorithm}

\cite{polson2013bayesian} developed an efficient method for sampling from PG distribution, which is the only nonstandard distribution involved in the above Gibbs sampler.

\section{Uniform ergodicity of the two-block Gibbs sampler }
\label{sec_proof}

In this section, we prove the uniform ergodicity of the two-block Gibbs sampler $\{\bm{\eta}^{(m)}, \bm{\omega}^{(m)}, \bm{\tau}^{(m)}\}_{m=0}^{\infty}$, which has the same rate of convergence
as the $\bm{\eta}$-marginal Markov chain $\{\bm{\eta}^{(m)}\}_{m=0}^{\infty}$
\citep{robe:rose:2001}. Below we analyze the $\bm{\Psi} \equiv \{\bm{\eta}^{(m)}\}_{m=0}^{\infty}$ chain.

Let $\bm{\eta}^{\prime}$ be the current state and $\bm{\eta}$ be the
next state, then the Markov transition density (Mtd) of $\bm{\Psi}$ is
\begin{equation}
\label{eq_mtd}
k(\bm{\eta}|\bm{\eta}^{\prime}) =  \int_{\mathbb{R}_+^r} \int_{\mathbb{R}_+^N} \pi(\bm{\eta}|\bm{\omega},\bm{\tau},\bm{y})\pi(\bm{\omega},\bm{\tau}|\bm{\eta}^{\prime},\bm{y})d\bm{\omega}d\bm{\tau},
\end{equation}
where $\pi(\cdot|\cdot,\bm{y})$'s are the conditional densities from section \ref{sec_gibbs}. Routine calculations show that
$k(\bm{\eta}|\bm{\eta}^{\prime})$ is reversible and thus invariant with respect to the marginal density of $\bm{\eta}$ denoted as $\pi(\bm{\eta}|\bm{y}) \equiv \int_{\mathbb{R}_+^r} \int_{\mathbb{R}_{+}^{N}} \pi(\bm{\eta},\bm{\omega},\bm{\tau}|\bm{y})d\bm{\omega}d\bm{\tau}$, where $\pi(\bm{\eta},\bm{\omega},\bm{\tau}|\bm{y})$ is defined in \eqref{eq:joint}. Since $k(\bm{\eta}|\bm{\eta}^{\prime})$ is strictly positive,  the Markov chain $\bm{\Psi}$ is Harris ergodic \citep{hobert2011data}.

Let $\mathscr{B}$ denote the Borel $\sigma$-algebra of $\mathbb{R}^{p+q}$ and
$K(\cdot,\cdot)$ be the Markov transition function corresponding to the Mtd
$k(\cdot, \cdot)$ in \eqref{eq_mtd}, that is, for any set
$A \in \mathscr{B}$, $\bm{\eta}^{\prime} \in \mathbb{R}^{p+q}$ and any
$j=0,1,\dots,$
\begin{equation}
\label{eq:def}
K(\bm{\eta}^{\prime}, A) = \mbox{Pr}(\bm{\eta}^{(j+1)} \in A | \bm{\eta}^{(j)} = \bm{\eta}^{\prime}) = \int_{A} k(\bm{\eta}| \bm{\eta}^\prime) d\bm{\eta}.
\end{equation}
Then the $m$-step Markov transition function is $K^m(\bm{\eta}^{\prime}, A) = \mbox{Pr}(\bm{\eta}^{(m+j)} \in A | \bm{\eta}^{(j)} = \bm{\eta}^{\prime})$. Let $\Pi(\cdot|\bm{y})$ be the probability measure with density
$\pi(\bm{\eta}|\bm{y})$. The Markov chain $\bm{\Psi}$ is geometrically ergodic if
there exists a constant $0 < t <1$ and a function
$G: \mathbb{R}^{p+q} \mapsto \mathbb{R}^+$ such that for any
$\bm{\eta} \in \mathbb{R}^{p+q}$,
\begin{equation}
\label{eq:ge}
||K^m(\bm{\eta},\cdot) - \Pi(\cdot|\bm{y})||:=\sup_{A\in \mathscr{B}} |K^m(\bm{\eta}, A) - \Pi(A|\bm{y})| \leq G(\bm{\eta}) t^m.
\end{equation}
If $G(\bm{\eta})$ is bounded above, then the corresponding Markov chain is uniformly ergodic. The following theorem establishes uniform ergodicity of the Markov chain $\bm{\Psi}$ by constructing a minorization condition.

\begin{theorem}
\label{them}
Assume that $a_j + q_j/2 \geq 1$ and $b_j >0$ for all $j$, then the Markov chain $\bm{\Psi}$ is uniformly ergodic.
\end{theorem}

\begin{proof}
We show that there exists a $\delta>0$ and a density function $h: \mathbb{R}^{p+q} \rightarrow [0,\infty)$ such that, for all $\bm{\eta}^{\prime}, \bm{\eta} \in \mathbb{R}^{p+q}$,
\begin{equation}
\label{eq_min}
k(\bm{\eta}|\bm{\eta}^{\prime}) \geq \delta h(\bm{\eta}).
\end{equation}

By \cite{roberts2004general}[Theorem 8], \eqref{eq_min} implies that the Markov chain $\bm{\Psi}$ is uniformly ergodic. Furthermore, under \eqref{eq_min}, \eqref{eq:ge} holds with $G=1$ and $t=1-\delta$.


For $\tau_j \geq \tau_0$, $j=1,\dots, q$, $\bm{\Sigma}\geq \bm{M}^{T}\bm{\Omega}\bm{M}+\bm{A}\left(\tau_{0}\right)$, that is $\bm{\Sigma} - (\bm{M}^{T}\bm{\Omega}\bm{M}+\bm{A}\left(\tau_{0}\right))$ is positive semidefinite. So
$\left|\bm{\Sigma}\right|\geq\left|\bm{M}^{T}\bm{\Omega}\bm{M}+\bm{A}\left(\tau_{0}\right)\right|=\left|\bm{A}\left(\tau_{0}\right)\right|\left|\tilde{\bm{M}}^{T}\bm{\Omega}\tilde{\bm{M}}+\bm{I}\right|\geq\left|\bm{A}\left(\tau_{0}\right)\right|,$
where $\tilde{\bm{M}}=\bm{M}\bm{A}\left(\tau_{0}\right)^{-1/2}$. And $
\bm{\mu}^{T}\bm{\Sigma}^{-1}\bm{\mu} \leq\bm{\mu}^{T}\left(\bm{M}^{T}\bm{\Omega}\bm{M}+\bm{A}\left(\tau_{0}\right)\right)^{-1}\bm{\mu} \leq\bm{\mu}^{T}\bm{A}\left(\tau_{0}\right)^{-1}\bm{\mu}$. Then, 
\begin{align*}
\pi(\bm{\eta}|\bm{\omega},\bm{\tau},\bm{y}) & =\left(2\pi\right)^{-\frac{p+q}{2}}\left|\bm{\Sigma}\right|^{\frac{1}{2}}\exp\left[-\frac{1}{2}\left(\bm{\eta}-\bm{\Sigma}^{-1}\bm{\mu}\right)^{T}\bm{\Sigma}\left(\bm{\eta}-\bm{\Sigma}^{-1}\bm{\mu}\right)\right]\\
& \geq\left(2\pi\right)^{-\frac{p+q}{2}}\left|\bm{A}\left(\tau_{0}\right)\right|^{1/2}\exp\left[-\frac{1}{2}\left(\bm{\eta}^{T}\bm{\Sigma}\bm{\eta}-2\bm{\eta}^{T}\bm{\mu}\right)-\frac{1}{2}\bm{\mu}^{T}\bm{A}\left(\tau_{0}\right)^{-1}\bm{\mu}\right]\\
& =\left(2\pi\right)^{-\frac{p+q}{2}}\left|\bm{A}\left(\tau_{0}\right)\right|^{1/2}\exp\left[-\frac{1}{2}\bm{\beta}^{T}\bm{Q}\bm{\beta}-\frac{1}{2}\bm{u}^{T}\bm{D}(\bm{\tau})\bm{u}+\bm{\eta}^{T}\bm{\mu}-\frac{1}{2}\bm{\mu}^{T}\bm{A}\left(\tau_{0}\right)^{-1}\bm{\mu}\right]\\
& \times\exp\left[-\frac{1}{2}\sum_{i=1}^{N}\omega_{i}\left(\bm{m}_{i}^{T}\bm{\eta}\right)^{2}\right].
\end{align*}

Therefore, 
\begin{align*}
& \pi(\bm{\eta}|\bm{\omega},\bm{\tau},\bm{y})\pi(\bm{\omega,\tau}|\bm{\eta}^{\prime},\bm{y})\\
& \geq\left(2\pi\right)^{-\frac{p+q}{2}}\left|\bm{A}\left(\tau_{0}\right)\right|^{1/2}\exp\left[-\frac{1}{2}\bm{\beta}^{T}\bm{Q}\bm{\beta}-\frac{1}{2}\bm{u}^{T}\bm{D}(\bm{\tau})\bm{u}+\bm{\eta}^{T}\bm{\mu}-\frac{1}{2}\bm{\mu}^{T}\bm{A}\left(\tau_{0}\right)^{-1}\bm{\mu}\right]\\
& \times\prod_{j=1}^{r}\frac{1}{c\left(\tau_{0},a_{j} + q_{j}/2, b_j +\bm{u}_{j}^{\prime T}\bm{u}_{j}^{\prime}/2\right)}\tau_{j}^{a_{j}+q_{j}/2-1}\exp\left[-\left(b_{j}+\bm{u}_{j}^{\prime T}\bm{u}_{j}^{\prime}/2\right)\tau_{j}\right]I(\tau_j \geq \tau_0)\\
& \times\prod_{i=1}^{N}\cosh^{n_i}\left(\frac{\left|\bm{m}_{i}^{T}\bm{\eta}^{\prime}\right|}{2}\right)\exp\left[-\frac{\left(\bm{m}_{i}^{T}\bm{\eta}^{\prime}\right)^{2}+\left(\bm{m}_{i}^{T}\bm{\eta}\right)^{2}}{2}\omega_{i}\right]p\left(\omega_{i}\right).
\end{align*}

According to \cite{polson2013bayesian} and \cite{choi2013polya}, 
\begin{align*}
 & \int_{\mathbb{R}_{+}}\exp\left[-\frac{\left(\bm{m}_{i}^{T}\bm{\eta}^{\prime}\right)^{2}+\left(\bm{m}_{i}^{T}\bm{\eta}\right)^{2}}{2}\omega_{i}\right]p\left(\omega_{i}\right)d\omega_{i} =  \left[\cosh\left(\frac{\sqrt{\left(\bm{m}_{i}^{T}\bm{\eta}^{\prime}\right)^{2}+\left(\bm{m}_{i}^{T}\bm{\eta}\right)^{2}}}{2}\right)\right]^{-n_{i}}\\
\geq & \left[\cosh\left(\frac{\left|\bm{m}_{i}^{T}\bm{\eta}^{\prime}\right|}{2}+\frac{ \left| \bm{m}_{i}^{T}\bm{\eta}\right|}{2}\right)\right]^{-n_{i}}
\geq  \left[2\cosh\left(\frac{\left|\bm{m}_{i}^{T}\bm{\eta}^{\prime}\right|}{2}\right)\cosh\left(\frac{\left|\bm{m}_{i}^{T}\bm{\eta}\right|}{2}\right)\right]^{-n_{i}},
\end{align*}
implying
\begin{align*}
 & \cosh^{n_{i}}\left(\frac{\left|\bm{m}_{i}^{T}\bm{\eta}^{\prime}\right|}{2}\right)\int_{\mathbb{R}_{+}}\exp\left[-\frac{\left(\bm{m}_{i}^{T}\bm{\eta}^{\prime}\right)^{2}+\left(\bm{m}_{i}^{T}\bm{\eta}\right)^{2}}{2}\omega_{i}\right]p\left(\omega_{i}\right)d\omega_{i} \geq 2^{-n_{i}}\cosh^{-n_{i}}\left(\frac{\left|\bm{m}_{i}^{T}\bm{\eta}\right|}{2}\right)\\
\geq & 2^{-n_{i}}\left[\exp\left(\frac{\left|\bm{m}_{i}^{T}\bm{\eta}\right|}{2}\right)\right]^{-n_{i}}\geq 2^{-n_{i}}\left[\exp\left(\frac{\left(\bm{m}_{i}^{T}\bm{\eta}\right)^{2}+1}{4}\right)\right]^{-n_{i}}
=  2^{-n_{i}}e^{-n_{i}/4}\exp\left[-\frac{n_{i}}{4}\left(\bm{m}_{i}^{T}\bm{\eta}\right)^{2}\right].
\end{align*}
So we have,
\begin{align*}
& \int_{\mathbb{R}_{+}^{N}}\pi(\bm{\eta}|\bm{\omega},\bm{\tau},\bm{y})\pi(\bm{\omega,\tau}|\bm{\eta}^{\prime},\bm{y})d\bm{\omega}\\
& \geq\left(2\pi\right)^{-\frac{p+q}{2}}\left|\bm{A}\left(\tau_{0}\right)\right|^{1/2}\exp\left[-\frac{1}{2}\bm{\beta}^{T}\bm{Q}\bm{\beta}+\bm{\eta}^{T}\bm{\mu}-\frac{1}{2}\bm{\mu}^{T}\bm{A}\left(\tau_{0}\right)^{-1}\bm{\mu}\right]\\
& \times 2^{-n}e^{-\frac{n}{4}}\exp\left[-\frac{1}{4}\bm{\eta}^{T}\bm{M}^{T}\bm{\Lambda}\bm{M}\bm{\eta}\right]\\
& \times\prod_{j=1}^{r}[c\left(\tau_{0},a_{j} + q_j/2,b_{j} +\bm{u}_{j}^{\prime T}\bm{u}_{j}^{\prime}/2\right)]^{-1}\tau_{j}^{a_{j}+q_{j}/2-1}\exp\left[-\left(b_{j}+\bm{u}_{j}^{\prime T}\bm{u}_{j}^{\prime}/2+\bm{u}_{j}^{T}\bm{u}_{j}/2\right)\tau_{j}\right]I(\tau_j \geq \tau_0),
\end{align*}
where $n = \sum_{i=1}^{N} n_i$ and $\bm{\Lambda}$ is the $N \times N$ diagonal matrix with $i$th diagonal element $n_i$. Thus, 
\begin{align*}
& k(\bm{\eta}|\bm{\eta}^{\prime }) = \int_{\mathbb{R}_+^{r}}\int_{\mathbb{R}_{+}^{N}}\pi(\bm{\eta}|\bm{\omega},\bm{\tau},\bm{y})\pi(\bm{\omega,\tau}|\bm{\eta}^{\prime},\bm{y})d\bm{\omega}d\bm{\tau}\\
& \geq\left(2\pi\right)^{-\frac{p+q}{2}}\left|\bm{A}\left(\tau_{0}\right)\right|^{1/2}\exp\left[-\frac{1}{2}\bm{\beta}^{T}\bm{Q}\bm{\beta}+\bm{\eta}^{T}\bm{\mu}-\frac{1}{2}\bm{\mu}^{T}\bm{A}\left(\tau_{0}\right)^{-1}\bm{\mu}\right]\times {2^{-n}e^{-\frac{n}{4}}\exp\left[-\frac{1}{4}\bm{\eta}^{T}\bm{M}^{T}\bm{\Lambda}\bm{M}\bm{\eta}\right]}\\
& \times \prod_{j=1}^{r}\frac{1}{c\left(\tau_{0},a_{j} + q_j/2,b_{j} + \bm{u}_{j}^{\prime T}\bm{u}_{j}^{\prime}/2\right)}\int_{\tau_{0}}^{\infty}\tau_{j}^{a_{j}+q_{j}/2-1}\exp\left[-\left(b_{j}+\bm{u}_{j}^{\prime T}\bm{u}_{j}^{\prime}/2+\bm{u}_{j}^{T}\bm{u}_{j}/2\right)\tau_{j}\right]d\tau_{j}.
\end{align*}
Now consider 
\begin{align*}
 & \frac{1}{c\left(\tau_{0},a_{j} + q_j/2,b_{j} +\bm{u}_{j}^{\prime T}\bm{u}_{j}^{\prime}/2\right)}\int_{\tau_{0}}^{\infty}\tau_{j}^{a_{j}+q_{j}/2-1}\exp\left[-\left(b_{j}+\bm{u}_{j}^{\prime T}\bm{u}_{j}^{\prime}/2+\bm{u}_{j}^{T}\bm{u}_{j}/2\right)\tau_{j}\right]d\tau_{j}\\
= & \frac{\left(b_{j}+\bm{u}_{j}^{\prime T}\bm{u}_{j}^{\prime}/2\right)^{a_{j}+q_{j}/2}}{\left(b_{j}+\bm{u}_{j}^{\prime T}\bm{u}_{j}^{\prime}/2+\bm{u}_{j}^{T}\bm{u}_{j}/2\right)^{a_{j}+q_{j}/2}}\cdot\frac{\int_{\left(b_{j}+\bm{u}_{j}^{\prime T}\bm{u}_{j}^{\prime}/2+\bm{u}_{j}^{T}\bm{u}_{j}/2\right)\tau_{0}}^{\infty}x^{a_{j}+q_{j}/2-1}\exp\left(-x\right)dx,}{\int_{\left(b_{j}+\bm{u}_{j}^{\prime T}\bm{u}_{j}^{\prime}/2\right)\tau_{0}}^{\infty}x^{a_{j}+q_{j}/2-1}\exp\left(-x\right)dx}
\end{align*}
For $x\geq 0 $, define,
$f_{1}\left(x\right)  =\int_{\left(b_{j}+x+\bm{u}_{j}^{T}\bm{u}_{j}/2\right)\tau_{0}}^{\infty}t^{a_{j}+q_{j}/2-1}\exp\left(-t\right)dt$, 
$f_{2}\left(x\right) =\int_{\left(b_{j}+x\right)\tau_{0}}^{\infty}t^{a_{j}+q_{j}/2-1}\exp\left(-t\right)dt$ and $g\left(x\right) =f_{1}\left(x\right)-\exp\left(-\tau_{0}\bm{u}_{j}^{T}\bm{u}_{j}/2\right)f_{2}\left(x\right)$. 
Since $a_{j}+q_{j}/2-1\geq 0$ by assumption, it can be shown that $g^{\prime}\left(x\right)\leq 0$. And $
g\left(x\right)\geq\lim_{x\rightarrow\infty}g\left(x\right)=0$. Thus $f_{1}\left(x\right)/f_{2}\left(x\right)\geq\exp\left[-\tau_{0}\bm{u}_{j}^{T}\bm{u}_{j}/2\right]$. Also $
(b_{j}+\bm{u}_{j}^{\prime T}\bm{u}_{j}^{\prime}/2)/(b_{j}+\bm{u}_{j}^{\prime T}\bm{u}_{j}^{\prime}/2+\bm{u}_{j}^{T}\bm{u}_{j}/2)\geq b_{j}/(b_{j}+\bm{u}_{j}^{T}\bm{u}_{j}/2)$, 
So 
\begin{align*}
\kappa (\bm{\eta} \vert \bm{\eta}^{\prime})& \geq\left(2\pi\right)^{-\frac{p+q}{2}}\left|\bm{A}\left(\tau_{0}\right)\right|^{1/2}\exp\left[-\frac{1}{2}\bm{\beta}^{T}\bm{Q}\bm{\beta}+\bm{\eta}^{T}\bm{\mu}-\frac{1}{2}\bm{\mu}^{T}\bm{A}\left(\tau_{0}\right)^{-1}\bm{\mu}\right]\\
& \times 2^{-n}e^{-\frac{n}{4}}\exp\left[-\frac{1}{4}\bm{\eta}^{T}\bm{M}^{T}\bm{\Lambda}\bm{M}\bm{\eta}\right] \times \prod_{j=1}^{r}\left(\frac{b_{j}}{b_{j}+\bm{u}_{j}^{T}\bm{u}_{j}/2}\right)^{a_{j}+q_{j}/2}\exp\left(-\tau_{0}\bm{u}_{j}^{T}\bm{u}_{j}/2\right).
\end{align*}

Let 
\begin{align*}
c_1\left(\bm{M},\bm{y}\right) & =\int_{\mathbb{R}^{p+q}}\exp\left[-\frac{1}{2}\bm{\beta}^{T}\bm{Q}\bm{\beta}+\bm{\eta}^{T}\bm{\mu}-\frac{1}{4}\bm{\eta}^{T}\bm{M}^{T}\bm{\Lambda}\bm{M}\bm{\eta}\right]\\
& \times \prod_{j=1}^{r}\left(\frac{b_{j}}{b_{j}+\bm{u}_{j}^{T}\bm{u}_{j}/2}\right)^{a_{j}+q_{j}/2}\exp\left(-\tau_{0}\bm{u}_{j}^{T}\bm{u}_{j}/2\right)d\bm{\eta}\\
& \leq \left(2\pi\right)^{\frac{p+q}{2}}\left|\bm{A}\left(\tau_{0}\right)\right|^{-1/2} \exp\left( \frac{1}{2} \bm{\mu}^T \bm{A}\left(\tau_{0}\right) \bm{\mu} \right)2^n e^{\frac{n}{4}}  < \infty.
\end{align*}
So there exists a density function $h\left(\bm{\eta}\right)$ and $\delta>0$ such that,
\[
k(\bm{\eta}|\bm{\eta}^{\prime}) = \int_{\mathbb{R}_+^{r}}\int_{\mathbb{R}_{+}^{n}}\pi(\bm{\eta}|\bm{\omega},\bm{\tau},\bm{y})\pi(\bm{\omega,\tau}|\bm{\eta}^{\prime},\bm{y})d\bm{\omega}d\bm{\tau}\geq\delta h\left(\bm{\eta}\right),
\]
where 
\begin{align*}
h\left(\bm{\eta}\right) & =\frac{1}{c_1\left(\bm{M},\bm{y}\right)}\exp\left[-\frac{1}{2}\bm{\beta}^{T}\bm{Q}\bm{\beta}+\bm{\eta}^{T}\bm{\mu}-\frac{1}{4}\bm{\eta}^{T}\bm{M}^{T}\bm{\Lambda}\bm{M}\bm{\eta}\right]\\
\times & \prod_{j=1}^{r}\left[b_{j}/(b_{j}+\bm{u}_{j}^{T}\bm{u}_{j}/2)\right]^{a_{j}+q_{j}/2}\exp\left(-\tau_{0}\bm{u}_{j}^{T}\bm{u}_{j}/2\right),\\
\text{and } \delta & =\left(2\pi\right)^{-\frac{p+q}{2}}\left|\bm{A}\left(\tau_{0}\right)\right|^{1/2}2^{-n}e^{-\frac{n}{4}}\cdot\exp\left[-\frac{1}{2}\bm{\mu}^{T}\bm{A}\left(\tau_{0}\right)^{-1}\bm{\mu}\right]c_1\left(\bm{M},\bm{y}\right).
\end{align*}

Hence the Markov chain is uniformly ergodic.

\end{proof}

\begin{remark}
	\label{rem:condition}
Since Theorem 1 does not put any conditions on $\bm{y}$, $\bm{X}$, $\bm{Z}$, $N$, $p$ and $q$, it is applicable in high dimensional situations where $p$ (or $q$) can be much larger than $N$.
\end{remark}

\begin{remark}
Following the proof of Theorem 1, the uniform ergodicity result in \cite{choi2013polya} can be extended to binomial data.
\end{remark}

\begin{remark}
Since $1/(x+1)\geq\exp\left(-x\right)$, we have  $b_j/(b_{j}+\bm{u}_{j}^{T}\bm{u}_{j}/2) \geq \exp\left[-\bm{u}_{j}^{T}\bm{u}_{j}/(2b_{j})\right]$. Using this inequality, we have 
\[\delta\geq2^{-n}e^{-\frac{n}{4}}\left|\bm{A}\left(\tau_{0}\right)\right|^{1/2}\left|\bm{\Sigma}_{1}\right|^{-1/2}\exp\left[-\frac{1}{2}\bm{\mu}^{T}\bm{A}\left(\tau_{0}\right)^{-1}\bm{\mu}+\frac{1}{2}\bm{\mu}^{T}\bm{\Sigma}_{1}^{-1}\bm{\mu}\right],
\]
where $\bm{\Sigma}_{1}=\frac{1}{2}\bm{M}^{T}\bm{\Lambda}\bm{M}+ \bm{Q}\oplus \left [\oplus_{j=1}^{r}\{\tau_{0} + (a_{j}+q_{j})/(2b_{j})\}\bm{I}_{q_{j}} \right] $. This in turn, gives a computable upper bound to the total variation distance to the stationary in \eqref{eq:ge}.
\end{remark}

\section{Discussion}
\label{sec_dis}
 We prove uniform ergodicity of the two-block Gibbs sampler for Bayesian logistic linear mixed models, which guarantees the existence of central limit theorem for MCMC estimators under a finite second moment condition \citep{jones2004markov}. Thus, our result has important practical implications as it allows for obtaining valid asymptotic standard errors for the posterior estimates \citep{flegal2010batch}. 
 Convergence rates analysis of Gibbs samplers for Bayesian logistic linear mixed models with improper priors is a potential  future project.

\bibliographystyle{apalike} 
\bibliography{ref_mixed_proper}

\end{document}